\documentclass[11pt]{amsart}
\usepackage[T1]{fontenc}
\usepackage{lmodern}
\usepackage[letterpaper,margin=1in]{geometry}
\usepackage{amsmath,amssymb,amsthm,mathtools}
\usepackage{microtype}
\usepackage{xcolor}
\usepackage[colorlinks=true,linkcolor=blue!40!black,citecolor=blue!40!black,urlcolor=blue!40!black]{hyperref}
\hypersetup{
 pdftitle={Polynomial growth of Bohnenblust--Hille constants on the Hamming cube},
 pdfauthor={Paata Ivanisvili, Haozhu Wang, Shengtong Zhang},
 pdfsubject={A weighted bootstrap proof of polynomial Bohnenblust--Hille bounds}}
\newtheorem{theorem}{Theorem}[section]
\newtheorem{corollary}[theorem]{Corollary}
\theoremstyle{remark}
\newtheorem{remark}[theorem]{Remark}
\numberwithin{equation}{section}
\newcommand{\E}{\mathbb E}
\newcommand{\bal}{\mathcal B}
\newcommand{\BH}{\mathsf B}
\allowdisplaybreaks[2]
\title[Polynomial Bohnenblust--Hille bounds on the cube]{Polynomial growth of Bohnenblust--Hille constants on the Hamming cube}
\author{Paata Ivanisvili}
\address{Department of Mathematics, University of California, Irvine, Irvine, CA 92697, USA}
\email{pivanisv@uci.edu}

\keywords{Bohnenblust--Hille inequality, Hamming cube, Walsh polynomials, hypercontractivity, random coordinate partitions}
\begin{document}
\begin{abstract}
We prove that the Bohnenblust--Hille constants for Walsh polynomials on the
Hamming cube grow at most polynomially in the degree. More precisely, there
is an absolute constant $K$ such that every $f :\{-1,1\}^{n} \to \mathbb{C}$ of
degree at most $m$ satisfies
\[
\left(\sum_{|S|\le m}|\widehat f(S)|^{2m/(m+1)}\right)^{(m+1)/(2m)}
\le Km^{27}\|f\|_\infty.
\]
The estimate is uniform in the dimension. The exponent $27$ is not
optimized.
\end{abstract}
\maketitle

\section{Introduction and main results}

Let $Q_n=\{-1,1\}^n$ be equipped with its uniform probability measure.
Every function $f:Q_n\to\mathbb C$ has the Fourier--Walsh expansion
\[
f(x)=\sum_{S\subseteq[n]}\widehat f(S)x^S,
\qquad x^S=\prod_{j\in S}x_j,
\qquad \widehat f(S)=\E f(x)x^S.
\]
We say that $f$ has degree at most $m$ if $\widehat f(S)=0$ whenever
$|S|>m$. Throughout the paper,
\[
q_r=\frac{2r}{r+1},\qquad
\frac1{q_r}=\frac12+\frac1{2r}\qquad(r\ge1).
\]
The Boolean Bohnenblust--Hille problem asks for the growth, as a function
of $m$, of the least dimension-independent constant $\BH_m$ in
\begin{equation}\label{eq:BH-intro}
\left(\sum_{|S|\le m}|\widehat f(S)|^{q_m}\right)^{1/q_m}
\le\BH_m\|f\|_\infty,
\qquad \deg f\le m.
\end{equation}

The classical Bohnenblust--Hille inequality arose in the solution of
Bohr's absolute convergence problem for Dirichlet series: the largest
possible gap between the abscissas of absolute and uniform convergence
is $1/2$ \cite{BH31}. Its polynomial form relates a coefficient
$\ell_{q_m}$ norm to a supremum norm on a polydisc. The dependence of the
inequality on the degree is a central quantitative question. In the
complex polynomial setting, Bayart, Pellegrino, and Seoane-Sep\'ulveda
\cite{BPS14} established subexponential estimates of the form
$\exp(O(\sqrt{m\log m}))$. They used these estimates to determine the
precise asymptotic behavior of the Bohr radius $K_n$ of the
$n$-dimensional polydisc, proving $K_n\sim\sqrt{(\log n)/n}$ as
$n\to\infty$ \cite{BPS14}.

For the Boolean cube, Defant, Masty\l o, and P\'erez \cite{DMP19} proved
the corresponding estimate
\begin{equation}\label{eq:DMP-bound}
\BH_m\le C\exp\!\bigl(c\sqrt{m\log m}\bigr)
\end{equation}
with absolute constants $C,c$.

Eskenazis and Ivanisvili \cite{EI22} used these constants $\BH_m$ to learn bounded
low-degree functions from a number of uniform random examples logarithmic
in the dimension.
Related Bohnenblust--Hille inequalities also yield learning bounds for
low-degree quantum observables \cite{VZ24}.

Recently, in a breakthrough paper, Pellegrino and Teixeira \cite{PT26}
obtained polynomial growth of the Bohnenblust--Hille constants for analytic
polynomials on the polydisc. Their argument uses a \emph{weighted square
function}. We adapt this approach to hamming cube $\{-1,1\}^{n}$. While approahces are similar there are certain distinctions between these proofs: on the polydisc, adjoining one variable makes a polynomial homogeneous
without changing its coefficient or supremum norm. On the cube, this
construction does not produce a homogeneous Walsh polynomial, and even a
single homogeneous level $f^{=r}$ can have larger supremum norm than $f$.
We therefore work with all degrees together in the weighted square
function $S(f)$. The random coloring of \cite{PT26} becomes simpler here:
each variable occurs at most once in a Walsh monomial, so a binomial
estimate captures, on average, a fixed proportion of every high-degree
coefficient sum. Bonami--Beckner hypercontractivity bounds the resulting
sums by square functions of coordinate restrictions. The weights then
give a contraction at fixed dimension, while Markov's inequality controls
the finitely many low degrees.

Our main result is the following.

\begin{theorem}\label{thm:main}
There is an absolute constant $K<\infty$ such that, for every pair of
positive integers $m,n$ with $m\le n$ and every $f:Q_n\to\mathbb C$ of degree
at most $m$,
\begin{equation}\label{eq:main}
\left(\sum_{|S|\le m}|\widehat f(S)|^{2m/(m+1)}\right)^{(m+1)/(2m)}
\le Km^{27}\|f\|_\infty.
\end{equation}
Equivalently, $\BH_m\le Km^{27}$ for every $m\ge1$.
\end{theorem}

The proof gives the following more structured estimate. Write
$f^{=r}=\sum_{|S|=r}\widehat f(S)x^S$ and
\[
\mathcal{S}_{r}(f)=\left(\sum_{|S|=r}|\widehat f(S)|^{q_r}\right)^{1/q_r}.
\]

\begin{theorem}\label{thm:weighted}
There is an absolute constant $K<\infty$ such that every polynomial as in
Theorem~\ref{thm:main} satisfies
\begin{equation}\label{eq:weighted-intro}
\left(\sum_{r=1}^m\frac{\mathcal{S}_{r}(f)^2}{r^{10}}\right)^{1/2}
\le Km^{22}\|f\|_\infty.
\end{equation}
\end{theorem}

In Section~\ref{sec:proof} we prove the two results together, starting
with the coefficient norm in \eqref{eq:main}.

As an application, Theorem~\ref{thm:weighted} establishes a special case of
the Aaronson--Ambainis conjecture \cite{AA14}, for polynomials whose Fourier
coefficients have constant absolute value on each degree level. For a
real-valued function $f$ on $Q_n$, write
\[
\operatorname{Var}(f)=\E|f-\E f|^2
=\sum_{S\ne\varnothing}|\widehat f(S)|^2,
\qquad
\operatorname{Inf}_i(f)=\sum_{S\ni i}|\widehat f(S)|^2.
\]

\begin{corollary}\label{cor:AA}
There is an absolute constant $c>0$ such that, for every $1\le m\le n$
and every $f:Q_n\to[-1,1]$ of degree at most $m$ satisfying
$|\widehat f(S)|=a_{|S|}$ for some $a_0,\ldots,a_m\ge0$ and all $|S|\le m$,
\[
\operatorname{Inf}_i(f)\ge
\frac{c}{m^{54}}\operatorname{Var}(f)^2
\qquad\text{for every }i\in[n].
\]
\end{corollary}

The assumption concerns only the magnitudes of the coefficients; their
signs may vary within each level. The proof of Corollary~\ref{cor:AA}
is given in Section~\ref{sec:AA}.

\section{Proof of the main results}\label{sec:proof}

\begin{proof}[Proof of Theorems~\ref{thm:main} and~\ref{thm:weighted}]
Fix positive integers $1\le m\le n$. If $f=0$, there is nothing to prove,
so by homogeneity we may assume $\|f\|_\infty=1$.
All empty sums and all layers above the degree of the polynomial are
understood to be zero. In particular, after freezing coordinates, the
degree bound $m$ may exceed the number of remaining coordinates; the same
convention applies to the resulting empty layers.

We first split the coefficient sum according to degree:

\[
\sum_{|S|\le m}|\widehat f(S)|^{q_m}
=|\widehat f(\varnothing)|^{q_m}
+\sum_{r=1}^m\sum_{|S|=r}|\widehat f(S)|^{q_m}.
\]
Write
\[
q_r=\frac{2r}{r+1},\qquad
\mathcal{S}_{r}(f)=\left(\sum_{|S|=r}|\widehat f(S)|^{q_r}\right)^{1/q_r}.
\]
Since $q_r\le q_m$, the coefficient norm on the $r$-th level is at most
$\mathcal{S}_{r}(f)$. Consequently,
\begin{align}
\left(\sum_{|S|\le m}|\widehat f(S)|^{q_m}\right)^{1/q_m}
&\le 1+\left(\sum_{r=1}^m \mathcal{S}_{r}(f)^{q_m}\right)^{1/q_m}\notag\\
&\le 1+m^{1/q_m-1/2}\left(\sum_{r=1}^m \mathcal{S}_{r}(f)^2\right)^{1/2}\notag\\
&\le 1+m^{5+1/(2m)}
\underbrace{\left(\sum_{r=1}^m\frac{\mathcal{S}_{r}(f)^2}{r^{10}}\right)^{1/2}}_{S(f)}.
\label{eq:start}
\end{align}
The first inequality uses $|\widehat f(\varnothing)|\le1$ and the
subadditivity of $t\mapsto t^{1/q_m}$, since $1/q_m\le1$.
The second is the comparison of normalized $\ell_{q_m}$ and $\ell_2$ norms:
\[
\left(\frac1m\sum_{r=1}^m \mathcal{S}_{r}(f)^{q_m}\right)^{1/q_m}
\le\left(\frac1m\sum_{r=1}^m \mathcal{S}_{r}(f)^2\right)^{1/2}.
\]

Thus we have reduced the problem to estimating
\[
S(f)=\left(\sum_{r=1}^m\frac{\mathcal{S}_{r}(f)^2}{r^{10}}\right)^{1/2}.
\]
It remains to prove $S(f)\le K m^{22}$. The reason for introducing the weights
$r^{-10}$ will become apparent when we split coordinates.

We may assume $m\ge12$: for $m<12$, the low-degree estimate below already
controls all of $S(f)$. Split the weighted sum:
\begin{equation}
S(f)\le
\underbrace{\left(\sum_{r=1}^{11}\frac{\mathcal{S}_{r}(f)^2}{r^{10}}\right)^{1/2}}_{H(f):\ \text{low degrees}}
+\underbrace{\left(\sum_{r=12}^{m}\frac{\mathcal{S}_{r}(f)^2}{r^{10}}\right)^{1/2}}_{T(f):\ \text{high degrees}}.
\label{eq:headtail}
\end{equation}

Estimating $H(f)$ is strightforward and follows from a standard fact 
$|f^{=r}(x)|\le\frac{m^{2r}}{r!} \|f\|_{\infty}$
We present its proof here for
completeness. For a fixed vertex $x$, consider
\[
p_x(t)=f(tx)=\sum_{r=0}^m t^r f^{=r}(x).
\]
Because $f$ is multilinear, its absolute value on $[-1,1]^n$ is bounded by its
maximum on the vertices. Indeed, with all other coordinates fixed,
$t\mapsto|A+Bt|$ is convex, so its maximum on $[-1,1]$ occurs at an endpoint.
Apply this successively in each coordinate. Hence
\[
|p_x(t)|\le1\qquad(-1\le t\le1).
\]
Extracting the $r$-th coefficient and applying the first-derivative Markov
inequality repeatedly gives (see \cite[Section~5.1]{BE95})
\[
|f^{=r}(x)|=\frac{|p_x^{(r)}(0)|}{r!}\le\frac{m^{2r}}{r!}.
\]

For each of the fixed degrees $1\le r\le11$, use a finite homogeneous
Bohnenblust--Hille constant $C_r$, whose existence follows from \cite{DMP19}:
\[
\mathcal{S}_{r}(f)\le C_r\|f^{=r}\|_\infty\le\frac{C_r m^{2r}}{r!}.
\]
Therefore, with a constant independent of $m,n$,
\begin{equation}
H(f)=\left(\sum_{r=1}^{11}\frac{\mathcal{S}_{r}(f)^2}{r^{10}}\right)^{1/2}
\le K m^{22}, \quad \text{where} \quad K=\left(\sum_{r=1}^{11}\frac{C_r^2}{r^{10}(r!)^2}\right)^{1/2}<\infty.
\label{eq:head}
\end{equation}

We now need to estimate the high-degree part $T(f)$. 

Keep $m,n$ fixed and let
\[
M=\sup_{\substack{\deg g\le m\\\|g\|_\infty\le1}}S(g)<\infty,
\]
where $g$ ranges over polynomials on the same $n$-dimensional cube; at this
stage, $M$ may depend on $m,n$.

We will show that, for every $f$ with $\|f\|_\infty=1$,
\[
T(f)=\left(\sum_{r=12}^m\frac{\mathcal{S}_{r}(f)^2}{r^{10}}\right)^{1/2}
\le\tfrac34M.
\]
Together with \eqref{eq:headtail} and \eqref{eq:head}, this will give
\[
S(f)\le H(f)+T(f)\le Km^{22}+\tfrac34M.
\]
For $0<\|g\|_\infty\le1$, apply this to $f=g/\|g\|_\infty$ and use
$S(g)=\|g\|_\infty S(f)$.
Thus the same bound holds for every $g$, such that $\|g\|_{\infty}\leq 1$ and $\mathrm{deg}(g) \leq m$. 
Taking the supremum over all such $g$'s gives
\[
M\le Km^{22}+\tfrac34M,
\]
and hence $M\le4Km^{22}$.

The class of polynomials of degree at most $m$ is stable under freezing
coordinates. Every resulting function still has degree at most $m$ and
supremum norm at most $1$. A function of fewer coordinates can be regarded
as a function on the original cube, independent of the unused coordinates.
Its weighted coefficient norm $S$, the quantity on the left-hand side of
\eqref{eq:weighted-intro}, is therefore bounded by the same $M$.

Return to the sum
\[
T(f)^2=\sum_{r=12}^m\frac{\mathcal{S}_{r}(f)^2}{r^{10}}.
\]
Split the coordinate indices $[n]=\{1,\ldots,n\}$ randomly into two disjoint
sets $I$ and $J$, with $I\cup J=[n]$, assigning each index to either set
with probability $1/2$, independently. Denote the resulting random partition
by $\omega=(I,J)$; $\E_\omega$ denotes averaging over these choices.
For a fixed partition, write $x=(x_i)_{i\in I}$ and $y=(y_j)_{j\in J}$;
thus $x$ now denotes only the coordinates in $I$, rather than the full
$n$-dimensional vector. With this notation,
\[
f(x,y)=\sum_{A\subset I,\ B\subset J}c_{A,B}x^Ay^B,
\qquad c_{A,B}=\widehat f(A\cup B).
\]
Each original coefficient $\widehat f(S)$ occurs exactly once, with
$A=S\cap I$ and $B=S\cap J$.

For a fixed set $S$ of size $r$, the number of its coordinates landing in $I$ is
\[
N_S=|S\cap I|=\sum_{j\in S}\mathbf1_{\{j\in I\}}.
\]
The indicators are independent Bernoulli variables with parameter $1/2$,
so $N_S$ has distribution $\operatorname{Bin}(r,1/2)$ and mean $r/2$.
Hoeffding's inequality \cite{Hoeffding63} therefore gives
\[
\begin{aligned}
p_r&:=\mathbb P_\omega\{r/6\le N_S\le5r/6\}\\
&=\mathbb P_\omega\{|N_S-r/2|\le r/3\}\\
&=1-\mathbb P_\omega\{|N_S-r/2|>r/3\}\\
&\ge1-2\exp\!\left(-\frac{2(r/3)^2}{r}\right)
=1-2e^{-2r/9}\\
&\ge p:=1-2e^{-8/3}>0.86\qquad(r\ge12).
\end{aligned}
\]
The probability $p_r$ depends only on $r$, and $p$ is a common lower bound
for all $r\ge12$.

For a fixed partition $\omega=(I,J)$ and $r\ge12$, define
$C_r(\omega)\ge0$ by
\[
C_r(\omega)^{q_r}
=\sum_{|S|=r}|\widehat f(S)|^{q_r}
\mathbf1_{\{r/6\le|S\cap I|\le5r/6\}}.
\]
Averaging the last identity gives
\[
\begin{aligned}
\E_\omega C_r(\omega)^{q_r}
&=\sum_{|S|=r}|\widehat f(S)|^{q_r}
\mathbb P_\omega\{r/6\le|S\cap I|\le5r/6\}\\
&=p_r \mathcal{S}_{r}(f)^{q_r}\ge p\,\mathcal{S}_{r}(f)^{q_r}.
\end{aligned}
\]
Since $q_r\le2$, we obtain
\[
\mathcal{S}_{r}(f)\le p^{-1/q_r}
\bigl(\E_\omega C_r(\omega)^{q_r}\bigr)^{1/q_r}
\le p^{-1/q_r}\bigl(\E_\omega C_r(\omega)^2\bigr)^{1/2}.
\]
For $r\ge12$, we also have $q_r\ge q_{12}=24/13$, so
$p^{-1/q_r}\le p^{-13/24}$ because $0<p<1$. Consequently,
\begin{align}
T(f)&=\left(\sum_{r=12}^m\frac{\mathcal{S}_{r}(f)^2}{r^{10}}\right)^{1/2}\notag\\
&\le\left(\sum_{r=12}^m\frac{p^{-2/q_r}}{r^{10}}
\E_\omega C_r(\omega)^2\right)^{1/2}\notag\\
&\le p^{-13/24}
\left(\E_\omega\sum_{r=12}^m\frac{C_r(\omega)^2}{r^{10}}\right)^{1/2}.
\label{eq:capture}
\end{align}

For a fixed partition $\omega=(I,J)$ and nonnegative integers $d,e$ with
$d+e\ge1$, put
\[
b_{d,e}=\left(
\sum_{\substack{A\subset I,\ |A|=d\\B\subset J,\ |B|=e}}
|c_{A,B}|^{q_{d+e}}\right)^{1/q_{d+e}}.
\]
Here $d$ and $e$ are the degrees in the two groups of variables; the
dependence of $b_{d,e}$ on $\omega$ is implicit.
On level $r\ge12$, the \emph{balanced pairs} are those with $d+e=r$
and $r/6\le d,e\le5r/6$. Grouping the terms in $C_r(\omega)$ by these
two degrees gives
\[
C_r(\omega)=
\left(\sum_{\substack{d+e=r\\r/6\le d,e\le5r/6}}
b_{d,e}^{q_r}\right)^{1/q_r}.
\]

There are at most $r+1$ blocks $b_{d,e}$ in $C_r$, since a pair of
nonnegative integers with $d+e=r$ is determined by $d\in\{0,\ldots,r\}$.
Since $1/q_r-1/2=1/(2r)$, comparison of the $\ell_{q_r}$ and $\ell_2$
norms of this sequence yields
\[
C_r\le(r+1)^{1/(2r)}
\left(\sum_{\substack{d+e=r\\r/6\le d,e\le5r/6}}b_{d,e}^2\right)^{1/2}.
\]
For $r\ge12$, the prefactor satisfies
$(r+1)^{1/(2r)}\le\kappa:=13^{1/24}$.
Write
\[
\bal=\{(d,e):12\le d+e\le m,\ (d+e)/6\le d,e\le5(d+e)/6\}.
\]
Substituting this bound for $C_r$ into \eqref{eq:capture} and reindexing
the sum by $r=d+e$ gives
\begin{align}
T(f)&\le p^{-13/24}\kappa
\left(\E_\omega\sum_{r=12}^m\frac1{r^{10}}
\sum_{\substack{d+e=r\\r/6\le d,e\le5r/6}}b_{d,e}^2\right)^{1/2}\notag\\
&=p^{-13/24}\kappa
\left(\E_\omega\sum_{(d,e)\in\bal}
\frac{b_{d,e}^2}{(d+e)^{10}}\right)^{1/2}.
\label{eq:blocks}
\end{align}

We have reached a coefficient sum over two blocks. Now comes the mixed-norm
step familiar from the standard Bohnenblust--Hille argument.

For integers $d,e\ge2$, put $r=d+e$ and write $\vartheta=d/r$. Define
\[
\begin{aligned}
X_{d,e}&=\left[
\sum_{|A|=d}\left(\sum_{|B|=e}|c_{A,B}|^2\right)^{q_d/2}
\right]^{1/q_d},\\
Y_{d,e}&=\left[
\sum_{|B|=e}\left(\sum_{|A|=d}|c_{A,B}|^2\right)^{q_e/2}
\right]^{1/q_e}.
\end{aligned}
\]
H\"older and Minkowski give
\[
b_{d,e}\le X_{d,e}^{d/(d+e)}Y_{d,e}^{e/(d+e)}.
\]
This is exactly Pellegrino and Teixeira's two-block coefficient inequality
\cite[Lemma~3.4]{PT26}. We present its short proof here for completeness:
\begin{align}
b_{d,e}&=\left(\sum_{|A|=d}\sum_{|B|=e}|c_{A,B}|^{q_r}\right)^{1/q_r}\notag\\
&\overset{\text{H\"older}}{\le}
\left[\sum_{|A|=d}\left(\sum_{|B|=e}|c_{A,B}|^2\right)^{q_d/2}\right]^{\vartheta/q_d}
\left[\sum_{|A|=d}\left(\sum_{|B|=e}|c_{A,B}|^{q_e}\right)^{2/q_e}\right]^{(1-\vartheta)/2}\notag\\
&\overset{\text{Minkowski}}{\le}X_{d,e}^{\vartheta}Y_{d,e}^{1-\vartheta}.
\label{eq:mix}
\end{align}
The first inequality applies H\"older in $B$ and then in $A$; the exponents
fit because
\[
\frac1{q_r}=\frac{\vartheta}{q_d}+\frac{1-\vartheta}{2}
=\frac{\vartheta}{2}+\frac{1-\vartheta}{q_e}.
\]
The second inequality uses Minkowski to bound
$\|c\|_{\ell_A^2(\ell_B^{q_e})}$ by
$\|c\|_{\ell_B^{q_e}(\ell_A^2)}=Y_{d,e}$, since $q_e\le2$.

Next, insert the \emph{damping factors} that hypercontractivity will need
later:
\[
\rho_d=\sqrt{\frac{d-1}{d+1}},\qquad
\widehat X_{d,e}=\rho_d^eX_{d,e},\qquad
\widehat Y_{d,e}=\rho_e^dY_{d,e}.
\]
Every balanced pair has $d,e\ge2$. Substituting these definitions into
\eqref{eq:mix} gives
\begin{align}
b_{d,e}&\le
\left(\frac{d+1}{d-1}\right)^{de/(2r)}
\left(\frac{e+1}{e-1}\right)^{de/(2r)}
\widehat X_{d,e}^{\vartheta}\widehat Y_{d,e}^{1-\vartheta}\notag\\
&\le3\,\widehat X_{d,e}^{\vartheta}\widehat Y_{d,e}^{1-\vartheta}.
\label{eq:damped}
\end{align}
For the last inequality, the power-series expansion
\[
k\log\frac{k+1}{k-1}
=2\sum_{j=0}^{\infty}\frac{1}{(2j+1)k^{2j}}
\qquad(k>1)
\]
shows that this expression decreases with $k$, so for $k\ge2$ it is at most
its value $2\log3$ at $k=2$. Taking the logarithm of the factor in
\eqref{eq:damped}, we obtain
\[
\begin{aligned}
\frac{de}{2r}\left(\log\frac{d+1}{d-1}+\log\frac{e+1}{e-1}\right)
&\le\frac{de}{2r}\left(\frac{2\log3}{d}+\frac{2\log3}{e}\right)\\
&=\log3.
\end{aligned}
\]

We now match the weight $(d+e)^{-10}$ in \eqref{eq:blocks} to the weights
$d^{-10}$ and $e^{-10}$, gaining a contraction factor.
Divide \eqref{eq:damped} by $r^5$:
\[
\frac{b_{d,e}}{r^5}\le
3\frac{d^{5\vartheta}e^{5(1-\vartheta)}}{r^5}
\left(\frac{\widehat X_{d,e}}{d^5}\right)^\vartheta
\left(\frac{\widehat Y_{d,e}}{e^5}\right)^{1-\vartheta}.
\]
The extra factor (recall that $\vartheta=d/r$) is
\[
\frac{d^{5\vartheta}e^{5(1-\vartheta)}}{r^5}
=\left[\vartheta^\vartheta(1-\vartheta)^{1-\vartheta}\right]^5.
\]
For a balanced pair, $\vartheta\in[1/6,5/6]$, and
\[
\vartheta^\vartheta(1-\vartheta)^{1-\vartheta}
\le\left(\frac{5^5}{6^6}\right)^{1/6}<\frac23.
\]
Indeed, the maximum occurs at the endpoints $1/6$ and $5/6$: the logarithm
$\vartheta\log\vartheta+(1-\vartheta)\log(1-\vartheta)$ is convex and symmetric.

Set
\[
U_{d,e}=\frac{\widehat X_{d,e}}{d^5},\qquad
V_{d,e}=\frac{\widehat Y_{d,e}}{e^5}.
\]
Then
\[
\frac{b_{d,e}}{r^5}\le3\left(\frac23\right)^5
U_{d,e}^\vartheta V_{d,e}^{1-\vartheta}.
\]
Square and use weighted AM--GM, followed by $\vartheta,1-\vartheta\le1$:
\begin{align}
\frac{b_{d,e}^2}{r^{10}}
&\le9\left(\frac23\right)^{10}
U_{d,e}^{2\vartheta}V_{d,e}^{2(1-\vartheta)}\notag\\
&\le9\left(\frac23\right)^{10}
\left(\vartheta U_{d,e}^2+(1-\vartheta)V_{d,e}^2\right)\notag\\
&\le9\left(\frac23\right)^{10}
\left(U_{d,e}^2+V_{d,e}^2\right).
\label{eq:young}
\end{align}
In view of \eqref{eq:blocks} and \eqref{eq:young}, it remains to bound the
sums of $U_{d,e}^2$ and $V_{d,e}^2$ over balanced pairs, uniformly in the
partition. Since all terms are nonnegative, we may bound the larger sums
over all $d,e\ge2$.

For the first sum, the definition $U_{d,e}=d^{-5}\widehat X_{d,e}$ gives
\[
\sum_{d,e\ge2}U_{d,e}^2
=\sum_{d\ge2}\frac1{d^{10}}\sum_{e\ge2}\widehat X_{d,e}^2.
\]
We first fix $d\ge2$ and estimate the inner sum over $e$; afterward we
restore the weight $d^{-10}$ and sum over $d$. Recall that
\[
\widehat X_{d,e}=\rho_d^eX_{d,e}
=\left[\sum_{|A|=d}
\left(\sum_{|B|=e}\rho_d^{2e}|c_{A,B}|^2\right)^{q_d/2}\right]^{1/q_d},
\]
where $\rho_d=\sqrt{(d-1)/(d+1)}$.
For each $A\subset I$ with $|A|=d$, collect the entire row
\[
R_A(y)=\sum_{B\subset J}c_{A,B}y^B.
\]
We keep all $y$-degrees together here. The noise operator is
$T_\rho(\sum_B u_By^B)=\sum_B\rho^{|B|}u_By^B$.
Minkowski, Parseval, and the Bonami--Beckner hypercontractive
inequality \cite{Bonami70,Beckner75} give the following chain:
\begin{align}
\left(\sum_{e\ge2}\widehat X_{d,e}^2\right)^{1/2}
&\le\left[\sum_{|A|=d}
\left(\sum_{B\subset J}\rho_d^{2|B|}|c_{A,B}|^2\right)^{q_d/2}
\right]^{1/q_d}\notag\\
&=\left(\sum_{|A|=d}\|T_{\rho_d}R_A\|_2^{q_d}\right)^{1/q_d}\notag\\
&\le\left(\sum_{|A|=d}\|R_A\|_{q_d}^{q_d}\right)^{1/q_d}\notag\\
&=\left(\E_y\sum_{|A|=d}|R_A(y)|^{q_d}\right)^{1/q_d}.
\label{eq:rows}
\end{align}
The first step uses the direction
$\ell_e^2(\ell_A^{q_d})\le\ell_A^{q_d}(\ell_e^2)$, valid since $q_d\le2$,
and then includes the nonnegative terms of $y$-degrees $0$ and $1$.
The hypercontractive step uses exactly $\rho_d^2=q_d-1$.
It applies to the full, possibly inhomogeneous polynomial $R_A$.

For fixed $y$, the numbers $R_A(y)$ are the degree-$d$ coefficients of
$x\mapsto f(x,y)$. Hence the last expression in \eqref{eq:rows} equals
\[
\left(\E_y \mathcal{S}_{d}(f(\,\cdot\,,y))^{q_d}\right)^{1/q_d}
\le\left(\E_y \mathcal{S}_{d}(f(\,\cdot\,,y))^2\right)^{1/2}.
\]
Square, divide by $d^{10}$, and sum in $d$:
\begin{align}
\sum_{d,e\ge2}U_{d,e}^2
&\le\E_y\sum_{d\ge2}
\frac{\mathcal{S}_{d}(f(\,\cdot\,,y))^2}{d^{10}}\notag\\
&\le M^2.
\label{eq:rowbound}
\end{align}
The last step holds because $f(\,\cdot\,,y)$ has degree at most $m$ and
supremum norm at most $1$. Similarly, interchanging $x$ and $y$ gives
\begin{equation}
\sum_{d,e\ge2}V_{d,e}^2\le M^2.
\label{eq:colbound}
\end{equation}

This explains why we introduced $M$: the coefficient sums produced
by hypercontractivity are precisely the weighted coefficient sums of
coordinate restrictions.

Combining \eqref{eq:young}, \eqref{eq:rowbound}, and \eqref{eq:colbound}, for
every partition,
\[
\sum_{(d,e)\in\bal}\frac{b_{d,e}^2}{(d+e)^{10}}
\le18\left(\frac23\right)^{10}M^2.
\]
Substitute this into \eqref{eq:blocks}:
\begin{equation}
T(f)\le
\underbrace{p^{-13/24}\kappa\,3\sqrt2\left(\frac23\right)^5}_{c}M \leq \frac{3}{4}M
\label{eq:tail}
\end{equation}

\begin{remark}
The balance condition $r/6 \leq d,e \leq 5r/6$ is what makes $c<1$: it bounds the weight factor
$[\vartheta^\vartheta(1-\vartheta)^{1-\vartheta}]^5$ by $(2/3)^5$,
enough to absorb the other losses. 
\end{remark}

Returning to our original split \eqref{eq:headtail}  we obtain
\[
M\le Km^{22}+\tfrac34M,
\]
and hence $M\le4Km^{22}$. By scaling (restoring homogeneity in $f$), we obtain
\begin{equation}
S(f)\le4Km^{22}\|f\|_\infty.
\label{eq:weighted}
\end{equation}
Everything was proved at fixed $n$, and the resulting constant is independent
of $n$. There is no need to assume in advance that the supremum over all
dimensions is finite. For $m<12$, the same low-degree estimate, with the sum
stopped at $r=m$, already gives the required bound.

Finally, return to the coefficient norm we started with. By \eqref{eq:start},
\begin{align*}
\left(\sum_{|S|\le m}|\widehat f(S)|^{q_m}\right)^{1/q_m}
&\le\left(1+4Km^{27+1/(2m)}\right)\|f\|_\infty\\
&\le K'm^{27}\|f\|_\infty,
\end{align*}
because $m^{1/(2m)}\le2$. Thus the argument gives
\[
\left(\sum_{|S|\le m}|\widehat f(S)|^{2m/(m+1)}\right)^{(m+1)/(2m)}
\le K'm^{27}\|f\|_\infty.
\]
This proves both theorems.
\end{proof}

\begin{remark}
We separate the first eleven levels so that the random-partition argument
applies uniformly to the remaining degrees. With our balance interval,
$r\ge12$ ensures $d,e\ge r/6\ge2$; hence $q_d,q_e>1$ and both damping
factors are positive, so they can be removed at a uniformly bounded cost.
The same cutoff gives the common probability bound $p=1-2e^{-8/3}$ used
in the contraction estimate.

Estimating the first eleven levels by Markov's inequality costs $m^{22}$,
and removing the weight $r^{-5}$ costs $m^5$. These two contributions give
the exponent $27$. For simplicity of the argument, we did not try to
optimize the exponent.
\end{remark}

\section{Application to the Aaronson--Ambainis conjecture}\label{sec:AA}

The Aaronson--Ambainis conjecture \cite[Conjecture~1.7]{AA14} asserts that
every $f:Q_n\to[-1,1]$ of degree at most $m$ has a coordinate $i$ with
$\operatorname{Inf}_i(f)\ge c_0(\operatorname{Var}(f)/m)^{C_0}$ for absolute
constants $c_0,C_0>0$. For coefficients of one common magnitude, the
implication from polynomial Bohnenblust--Hille bounds was already noted
in \cite{DMP19}; here the magnitude may vary with the degree.

\begin{proof}[Proof of Corollary~\ref{cor:AA}]
Put $v_r=\sum_{|S|=r}|\widehat f(S)|^2=\binom nr a_r^2$. The hypothesis
gives, for every $i\in[n]$,
\[
\operatorname{Var}(f)=\sum_{r=1}^m v_r,
\qquad
\operatorname{Inf}_i(f)=\sum_{r=1}^m\binom{n-1}{r-1}a_r^2
=\frac1n\sum_{r=1}^m rv_r.
\]
Moreover, since $2/q_r=1+1/r$ and $\binom nr\ge(n/r)^r$,
\[
\mathcal{S}_r(f)^2=\binom nr^{1/r}v_r\ge\frac nr\,v_r.
\]
Cauchy--Schwarz and Theorem~\ref{thm:weighted} therefore yield
\begin{align*}
\operatorname{Var}(f)^2
&\le\left(\sum_{r=1}^m rv_r\right)
       \left(\sum_{r=1}^m\frac{v_r}{r}\right)\\
&\le\operatorname{Inf}_i(f)\sum_{r=1}^m\mathcal{S}_r(f)^2\\
&\le m^{10}\operatorname{Inf}_i(f)
       \sum_{r=1}^m\frac{\mathcal{S}_r(f)^2}{r^{10}}\\
&\le K^2m^{54}\operatorname{Inf}_i(f),
\end{align*}
where the last step uses $\|f\|_\infty\le1$.
Taking $c=K^{-2}$ proves the corollary.
\end{proof}

\section*{Acknowledgments}
P.~Ivanisvili was supported in part by NSF grants DMS-2152401 (CAREER)
and DMS-2554183, a Simons Fellowship, and a Humboldt Research Fellowship
for Experienced Researchers.

The authors acknowledge the use of AI tools, in particular Grok Bots.
Following the release of Pellegrino and Teixeira's paper \cite{PT26} on
polynomial growth of the complex polynomial Bohnenblust--Hille constants,
we prompted Grok Bots to use that work as a starting point for proving
polynomial bounds on the Hamming cube. Grok Bots generated a proof,
which we checked and verified to be correct. Grok's raw output was
publicly circulated the author in an X post on August~27, 2026;
see the original draft \cite{Grok26}. The present paper develops
that proof into a detailed exposition, with additional explanations to make the argument easier to follow.

\end{document}